\documentclass{amsart} %{amsart}%{article}%{elsart}
\usepackage{graphicx}
\usepackage{amssymb, amsmath, sidecap}
\usepackage{amsfonts}
\usepackage{amssymb}
\usepackage{float}
\usepackage{amsfonts}
\usepackage{extarrows}
\usepackage{mathrsfs}
\usepackage{booktabs}
\usepackage{verbatim}
\usepackage{hyperref}
\usepackage[usenames,dvipsnames]{xcolor}

\renewcommand{\baselinestretch}{1}

\def\bt{\begin{thm}}
\def\et{\end{thm}}
\def\bl{\begin{lem}}
\def\el{\end{lem}}
\def\bd{\begin{defi}}
\def\ed{\end{defi}}
\def\bc{\begin{cor}}
\def\ec{\end{cor}}
\def\bp{\begin{proof}}
\def\ep{\end{proof}}
\def\br{\begin{rem}}
\def\er{\end{rem}}

\def\be{\begin{equation}}
\def\ee{\end{equation}}
\def\bes{\begin{equation*}}
\def\ees{\end{equation*}}
\def\bea{\begin{equation} \begin{aligned}}
\def\eea{\end{aligned} \end{equation}}
\def\beas{\begin{equation*} \begin{aligned}}
\def\eeas{\end{aligned} \end{equation*}}
\def\ba{\begin{align}}
\def\ea{\end{align}}
\def\bas{\begin{align*}}
\def\eas{\end{align*}}

\newtheorem{thm}{Theorem}[section]
\newtheorem{lem}{Lemma}[section]
\newtheorem{defi}{Definition}[section]
\newtheorem{ex}{Example}[section]
\newtheorem{prop}[thm]{Proposition}

\newtheorem{rem}{Remark}[section]
\newtheorem{cor}{Corollary}[section]

\numberwithin{equation}{section}
\numberwithin{figure}{section}

\begin{document}
\title{On a compact trace embedding theorem in Musielak-Sobolev spaces
\footnotemark[1]}

\author[Wang]{Li Wang}
\address[Li Wang]{School of Mathematics and Statistics, Lanzhou
University, Lanzhou 730000, P. R. China} \email{lwang10@lzu.edu.cn}

\author[Liu]{Duchao Liu}
\address[Duchao Liu]{School of Mathematics and Statistics, Lanzhou
University, Lanzhou 730000, P. R. China} \email{liuduchao@gmail.com,
Tel.: +8613893289235, fax: +8609318912481}

\footnotetext[1]{
Research supported by the National Natural Science Foundation of
China (NSFC 11501268 and NSFC 11471147) and the Fundamental Research Funds for the 
Central Universities (lzujbky-2016-101).}

\begin{abstract}
By a stronger compact boundary embedding theorem in Musielak-Orlicz-Sobolev space developed in the paper, variational method is employed to deal with the nonlinear elliptic equation with the nonlinear Neumann boundary condition in the framework of Musielak-Orlicz-Sobolev space.
\end{abstract}
 \maketitle

%\tableofcontents

\section{Introduction}

In this paper we develop a stronger compact trace embedding of Sobolev type in Musielak-Orlicz-Sobolev (or Musielak-Sobolev) spaces than the existed results in mathematics literature \cite{Liu:16}. And as an application of this stronger compact trace embedding theorem, the existence and multiplicity of weak solutions in Musielak-Orlicz-Sobolev spaces for the following equation are considered
\begin{equation*}
\begin{cases}
  -\text{div}(a_{1}(x,|\nabla u|)\nabla u)+a_{0}(x,u)=0 &\text{ in } \Omega,\\
a_{1}(x,|\nabla u|)\frac{\partial u}{\partial n}=g(x,u) &\text{ on } \partial \Omega,
\end{cases}
\end{equation*}
where $\Omega$ is a bounded smooth domain in $\mathbb{R}^{N}$, $a_{0},a_1:\Omega\times \mathbb{R}\to \mathbb{R}$ are Carath\'{e}odory functions and $\frac{\partial u}{\partial n}$ is the outward unit normal derivative on $\partial \Omega$.

In the study of nonlinear differential equations, it is well known
that more general functional space can handle differential equations
with more complex nonlinearities. If we want to study a general form
of differential equations, it is very important to find a proper
functional space in which the solutions may exist. Musielak-Sobolev 
space is such a general kind of Sobolev space that the classical
Sobolev spaces, Orlicz-Sobolev
spaces and variable exponent Sobolev spaces can be interpreted as its special cases.

Differential equations in Orlicz-Sobolev spaces have been studied extensively in recent years (see \cite{Wang_lh, Trudinger_1971, Eleuteri, Fan_vari_boundary, Fan3, Mihai}). To the best of our knowledge, however, differential equations in Musielak-Sobolev spaces have been studied little. In \cite{Benkirane:13}, Benkirane and Sidi give an embedding theorem in Musielak-Sobolev spaces and an existence result for nonlinear elliptic equations.
In \cite{Fan:12} and \cite{1Fan:12}, Fan gives two embedding theorems and some properties of differential operators in Musielak-Sobolev spaces.
In the research of the paper \cite{Liu:15}, Liu and Zhao give existence the results for nonlinear elliptic equations with homogeneous Dirichlet boundary. In a much recent paper \cite{Liu:16}, Liu, Wang and Zhao give a trace embedding theorem and a compact trace embedding theorem in bounded domain about this kind of functional space. Motivated by \cite{Liu:16}, we find that the results in that paper can be better improved and it is possible to make an application of the result to get some existence results for nonlinear elliptic equations with nonlinear Neumann boundary condition. By developing a stronger trace embedding theorem in Musielak-Sobolev spaces, our aim in this paper is to study the existence of solutions to a kind of elliptic differential equation with Neumann boundary condition under the functional frame work of Musielak-Sobolev spaces, which is a more general case in the variable exponent Sobolev spaces and Orlicz-Sobolev spaces.

As an example of Musielak-Sobolev spaces we claim that not only variable exponent Sobolev spaces satisfy the conditions in our theorem (see \textbf{Example \ref{ex1}} in Section \ref{Sec2}), but some more complex case can also work (see \textbf{Example \ref{ex2}} in Section \ref{Sec2}). And the current existed theories in mathematical literature can not cover the case in \textbf{Example \ref{ex2}}.

The paper is organized as follows. In Section \ref{Sec2},  we develop a stronger compact trace embedding theorem and for the
readers' convenience recall some definitions and properties about
Musielak-Orlicz-Sobolev spaces. In Section \ref{Sec3}, we give some propositions of differential operators in Musielak-Sobolev spaces. In Section \ref{Sec4},
we give the existence results of weak solutions to the equations.

\section{A stronger compact trace embedding theorem in Musielak-Orlicz-Sobolev spaces}\label{Sec2}

In this section, we will give a stronger compact trace embedding theorem to proceed.  And we will list some basic definitions and propositions about Musielak-Orlicz-Sobolev spaces.

Firstly, we give the definition of \textit{N-function} and \textit{generalized N-function} as follows.

\begin{defi}
A function $A:\mathbb{R}\to[0,+\infty)$ is called an $N$-function, denoted by $A\in N$, %if $A$ is even and convex, $A(0)=0$, $0<A(t)\in C^{0}$ for $t>0$, and the following conditions hold
\begin{equation*}
\lim\limits_{t\to0+}\frac{A(t)}{t}=0\text{ and }\lim\limits_{t\to+\infty}\frac{A(t)}{t}=+\infty.
\end{equation*}
A function $A:\Omega\times\mathbb{R}\to[0,+\infty)$ is called a generalized $N$-function, denoted by $A\in N(\Omega)$, if for each $t\in [0,+\infty)$, the function $A(\cdot,t)$ is measurable, and for a.e. $x\in \Omega$, $A(x,\cdot)\in N$.
\end{defi}

Let $A\in N(\Omega)$, the Musielak-Orlicz space $L^{A}(\Omega)$ is defined by
\begin{equation*}
\begin{aligned}
L^{A}(\Omega):=\{u:u\text{ is a measurable real function,}&\\ 
\text{and }\exists \lambda>0\text{ such that }&\int\limits_{\Omega}A\bigg(x,\frac{|u(x)|}{\lambda}\bigg)\,\mathrm{d}x<+\infty\}
\end{aligned}
\end{equation*}
with the (\textit{Luxemburg}) norm
\begin{equation*}
\|u\|_{L^{A}(\Omega)}=\|u\|_{A}:=
\inf\bigg\{\lambda>0:\int\limits_{\Omega}A\bigg(x,\frac{|u(x)|}{\lambda}\bigg)\,\mathrm{d}x\leq1\bigg\}.
\end{equation*}
The Musielak-Sobolev space $W^{1,A}(\Omega)$ can be defined by
\begin{equation*}
W^{1,A}(\Omega):=\{u\in L^{A}(\Omega):|\nabla u|\in L^{A}(\Omega)\}
\end{equation*}
with the norm
\begin{equation*}
\|u\|:=\|u\|_{W^{1,A}(\Omega)}:=\|u\|_{A}+\|\nabla u\|_{A},
\end{equation*}
where $\|\nabla u\|_{A}:=\||\nabla u|\|_{A}$. And $\|u\|_{\rho}:=\inf\{\lambda>0,\int\limits_{\Omega}A(x,\frac{|u|}{\lambda})+A(x,\frac{|\nabla u|}{\lambda}\big)\,\mathrm{d}x\leq 1\}$ is an equivalent norm on $W^{1,A}(\Omega)$, see in \cite{Fan:12}.

We say that $A\in N(\Omega)$ satisfies Condition $\Delta_{2}(\Omega)$, if there exists a positive constant $K>0$ and a nonnegative function $h\in L^{1}(\Omega)$ such that
\begin{equation*}
A(x,2t)\leq K A(x,t)+h(x)\quad\text{for }x\in\Omega \text{ and }t\in \mathbb{R}.
\end{equation*}

$A$ is called locally integrable if $A(\cdot,t_{0})\in L_{\text{loc}}^{1}(\Omega)$ for every $t_{0}>0$. For $x\in\Omega$ and $t\geq0$, we denote by $a(x,t)$ the right-hand derivative of $A(x,t)$ at $t$, at the same time define $a(x,t)=-a(x,-t)$. Then $A(x,t)=\int_{0}^{|t|}a(x,s)\,\mathrm{d}s$ for $x\in \Omega$ and $t\in R$.

Define $\widetilde{A}:\Omega\times \mathbb{R}\to[0,+\infty)$ by
\begin{equation*}
\widetilde{A}(x,s)=\sup\limits_{t\in \mathbb{R}}(st-A(x,t))\text{ for }x\in\Omega\text{ and }s\in \mathbb{R}.
\end{equation*}
$\widetilde{A}$ is called the complementary function to $A$ in the sense of \textit{Young}. It is well known that $\widetilde{A}\in N(\Omega)$ and $A$ is also the complementary function to $\widetilde A$.

For $x\in \Omega$ and $s\geq 0$, we denote by $a_{+}^{-1}(x,s)$ the right-hand derivative of $\widetilde{A}(x,\cdot)$ at $s$, at the same time define $a_{+}^{-1}(x,s)=-a_{+}^{-1}(x,-s)$ for $x\in \Omega$ and $s\leq0$. Then for $x\in \Omega$ and $s\geq 0$, we have
\begin{equation*}
a_{+}^{-1}(x,s)=\sup\{t\geq0:a(x,t)\leq s\}=\inf\{t>0:a(x,t)>s\}.
\end{equation*}

\begin{prop}(See \cite{Liu:15}.)
Let $A\in N(\Omega)$. Then the following assertions hold.
\begin{enumerate}
\item[(1)] $A(x,t)\leq a(x,t)t\leq A(x,2t)$ for $x\in\Omega$ and $t\in \mathbb{R}$;
\item[(2)] $A$ and $\widetilde{A}$ satisfy the Young inequality
\begin{equation*}
st\leq A(x,t)+\widetilde{A}(x,s) \text{ for } x\in\Omega\text{ and }s,t\in \mathbb{R}
\end{equation*}
and the equality holds if $s=a(x,t)$ or $t=a_{+}^{-1}(x,s)$.
\end{enumerate}
\end{prop}

Let
$A,B\in N(\Omega)$. We say that $A$ is weaker than $B$, denoted by $A\preccurlyeq B$, if there exist positive constants $K_{1},K_{2}$ and a nonnegative function $h\in L^{1}(\Omega)$ such that
\begin{equation*}
A(x,t)\leq K_{1}B(x,K_{2}t)+h(x)\text{ for }x\in\Omega\text{ and }t\in[0,+\infty).
\end{equation*}

\begin{prop}(See \cite{Liu:15}.)
Let $A,B\in N(\Omega)$, and $A\preccurlyeq B$. Then $\widetilde B\preccurlyeq \widetilde{A}$, $L^{B}(\Omega)\hookrightarrow L^{A}(\Omega)$ and $L^{\widetilde{A}}(\Omega)\hookrightarrow L^{\widetilde B}(\Omega)$.
\end{prop}

\begin{prop}\label{p2.3}(See \cite{Liu:15}.)
Let $A\in N(\Omega)$ satisfy $\Delta_{2}(\Omega)$. Then the following assertions hold,
\begin{enumerate}
\item[(1)] $L^{A}(\Omega)=\{u:u$ is a measurable function, and $\int_{\Omega}A(x,|u(x)|)\,\mathrm{d}x<+\infty\}$;
\item[(2)] $\int_{\Omega}A(x,|u(x)|)\,\mathrm{d}x<1$ (resp. $=1$; $>1$) $\Leftrightarrow \|u\|_{A}<1$ (resp. $=1$; $>1$), where $u\in L^{A}(\Omega)$;
\item[(3)] $\int_{\Omega}A(x,|u_{n}(x)|)\,\mathrm{d}x\to0$ (resp. $1$; $+\infty$) $\Leftrightarrow \|u_{n}\|_{A}\to0$ (resp. $1$; $+\infty$), where $\{u_{n}\}\subset L^{A}(\Omega)$;
\item[(4)] $u_{n}\to u$ in $L^{A}(\Omega)$ $\Rightarrow$ $\int_{\Omega}|A(x,|u_{n}(x)|)-A(x,|u(x)|)|\,\mathrm{d}x\to 0$ as $n\to\infty$;
\item[(5)] If $\widetilde{A}$ also satisfies $\Delta_{2}(\Omega)$, then
\begin{equation*}
\bigg|\int_{\Omega}u(x)v(x)\,\mathrm{d}x\bigg|\leq2\|u\|_{A}\|v\|_{\widetilde{A}},\qquad\forall\, u\in L^{A}(\Omega), v\in L^{\widetilde{A}}(\Omega);
\end{equation*}
\item[(6)] $a(\cdot,|u(\cdot)|)\in L^{\widetilde{A}}(\Omega)$ for every $u\in L^{A}(\Omega)$.
\end{enumerate}
\end{prop}

\begin{prop}(See \cite{Liu:15}.)
Let $A\in N(\Omega)$ be locally integrable. Then $(L^{A}(\Omega),\|\cdot\|)$ is a separable Banach space, $W^{1,A}(\Omega)$ is reflexive provided $L^{A}(\Omega)$ is reflexive.
\end{prop}

The following assumptions on $A\in N(\Omega)$ will be used.
\begin{enumerate}
\item[$(A_{1})$] $\inf\limits_{x\in\Omega}A(x,1)=c_{1}>0$;
\item[$(A_{2})$] For every $t_{0}>0$, there exists $c=c(t_{0})>0$ such that
\begin{equation*}
\frac{A(x,t)}{t}\geq c\text{ and }\frac{\widetilde{A}(x,t)}{t}\geq c,\forall\, t\geq t_{0},x\in\Omega.
\end{equation*}
\end{enumerate}

\begin{rem}
It is easy to see that $(A_{2})\Rightarrow(A_{1})$.
\end{rem}

\begin{prop}\label{embedd}(See \cite{Liu:15}.)
If $A\in N(\Omega)$ satisfies $(A_{1})$, then $L^{A}(\Omega)\hookrightarrow L^{1}(\Omega)$ and $W^{1,A}(\Omega)\hookrightarrow W^{1,1}(\Omega)$.
\end{prop}

\begin{prop}\label{p2.6}(See \cite{Fan:12}.)
Let $A\in N(\Omega)$, both $A$ and $\widetilde{A}$ be locally integrable and satisfy $\Delta_{2}(\Omega)$ and $(A_{2})$. Then $L^{A}(\Omega)$ is reflexive, and the mapping $J:L^{\widetilde{A}}(\Omega)\to (L^{A}(\Omega))^{*}$ defined by
\begin{equation*}
\langle J(v),w\rangle=\int_{\Omega}v(x)w(x)\,\mathrm{d}x,\quad\forall\, v\in L^{\widetilde{A}}(\Omega),w\in L^{A}(\Omega)
\end{equation*}
is a linear isomorphism and $\|J(v)\|_{(L^{A}(\Omega))^{*}}\leq 2\|v\|_{L^{\widetilde{A}}(\Omega)}$.
\end{prop}

To give the embedding theorems, the following assumptions are introduced.
\begin{enumerate}
\item[(P1)] $\Omega\subset \mathbb{R}^{N}(N\geq2)$ is a bounded domain with the cone property, and $A\in N(\Omega)$;
\item[(P2)] $A:\overline{\Omega}\times \mathbb{R}\to[0,+\infty)$ is continuous and $A(x,t)\in(0,+\infty)$ for $x\in\overline{\Omega}$ and $t\in(0,+\infty)$.
\item[(P3)] Suppose $A\in N(\Omega)$ satisfies
\begin{equation}
\int_0^1\frac{A^{-1}(x,t)}{t^{\frac{n+1}{n}}}\mathrm{d} t<+\infty,\,\forall x\in\overline\Omega,
\end{equation}
replacing, if necessary, $A$ by another $N(\Omega)$-function equivalent to $A$ near infinity.
\end{enumerate}

Under assumption (P1), (P2) and (P3), for each $x\in\overline{\Omega}$ we can define $A_{*}(x,t)$ for $x\in\Omega$ and $t\in\mathbb{R}^N$ as in \cite{Liu:16}. We call $A_{*}$ the Sobolev conjugate function of $A$.  Set
\begin{equation}\label{T(x)}
 T(x):= \lim\limits_{s\to+\infty}A_{*}^{-1}(x,s).
\end{equation}
Then $A_{*}(x,\cdot)\in C^{1}(0,T(x))$. Furthermore for $A\in N(\Omega)$ and $T(x)=+\infty$ for any $x\in\overline{\Omega}$, it is known that $A_{*}\in N(\Omega)$ (see in \cite{Adams:03}).

Let $X$ be a metric space and $f:X\rightarrow(-\infty,+\infty]$ be
an extended real-valued function. For $x\in X$ with $f(x)\in
\mathbb{R}$, the continuity of $f$ at $x$ is well defined. For $x\in
X$ with $f(x)=+\infty$, we say that $f$ is continuous at $x$ if
given any $M>0$, there exists a neighborhood $U$ of $x$ such that
$f(y)>M$ for all $y\in U$. We say that
$f:X\rightarrow(-\infty,+\infty]$ is continuous on $X$ if $f$ is
continuous at every $x\in X$. Define Dom$(f)=\{x\in X :
f(x)\in\mathbb{R}\}$ and denote by $C^{1-0}(X)$ the set of all
locally Lipschitz continuous real-valued functions defined on $X$.

The following assumptions will also be used.
\begin{enumerate}
\item[(P4)] $T:\overline{\Omega}\to[0,+\infty]$ is continuous on $\overline{\Omega}$ and $T\in C^{1-0}(\text{Dom}(T))$;
\item[(P5)] $A_{*}\in C^{1-0}(\text{Dom}(A_{*}))$ and there exist positive constants $\delta_{0}<\frac{1}{N}, C_{0}$ and $0<t_{0}<\min\limits_{x\in\overline{\Omega}}T(x)$ such that
\begin{equation*}
|\nabla_{x}A_{*}(x,t)|\leq C_{0}(A_{*}(x,t))^{1+\delta_{0}},\quad j=1,\cdots,N,
\end{equation*}
for $x\in \Omega$ and $|t|\in[t_{0},T(x))$ provided $\nabla_{x}A_{*}(x,t)$ exists.
\end{enumerate}

Let $A,B\in N(\Omega)$, we say that $A\ll B$ if, for any $k>0$,
\begin{equation*}
\lim\limits_{t\to+\infty}\frac{A(x,kt)}{B(x,t)}=0\text{ uniformly for }x\in\Omega.
\end{equation*}

 The following theorems are embedding theorems for Musielak-Sobolev spaces developed by Fan in \cite{1Fan:12} and Liu and Zhao in \cite{Liu:16}.
\begin{thm}(See \cite{1Fan:12}.)
Let (P1)-(P5) holds. Then
\begin{enumerate}
\item[(1)] There is a continuous embedding $W^{1,A}(\Omega)\hookrightarrow L^{A_{*}}(\Omega)$;
\item[(2)] Suppose that $B\in N(\Omega)$, $B:\overline{\Omega}\times[0,+\infty)\to[0,+\infty)$ is continuous, and $B(x,t)\in (0,+\infty)$ for $x\in\Omega$ and $t>0$. If $B\ll A_{*}$, then there is a compact embedding $W^{1,A}(\Omega)\hookrightarrow\hookrightarrow L^{B}(\Omega)$.
\end{enumerate}
\end{thm}

To give the boundary trace embedding theorem, we need the following assumption:
\begin{enumerate}
\item[$(P_{\partial}^{1})$] $A$ is an $N(\Omega)$-function such that $T(x)$ defined in Equation (\ref{T(x)}) satisfies
\begin{equation}\label{P^{1}}
T(x)=+\infty, \forall x\in \overline{\Omega}.
\end{equation}
\end{enumerate}

The following theorem is trace embedding theorem in \cite{Liu:16}.

\begin{thm}\label{liu1}(See \cite{Liu:16}, Theorem 4.2.)
Let $\Omega\subset \mathbb{R}^{N}$ be an open bounded domain with Lipschitz boundary. Suppose that $A\in N(\Omega)$ satisfies $(P1)$-$(P5)$ and $(P^{1}_{\partial})$. Then there is a continuous boundary trace embedding $W^{1,A}(\Omega)\hookrightarrow L^{A_{*}^{\frac{N-1}{N}}}(\partial\Omega)$, in which $A_{*}^{\frac{N-1}{N}}(x,t)=[A_{*}(x,t)]^{\frac{N-1}{N}},\forall x\in\partial \Omega,t\in \mathbb{R}^{+}$.
\end{thm}

The following theorem is a compact trace embedding theorem in \cite{Liu:16}.

\begin{thm}\label{liu2}(See \cite{Liu:16}, Theorem 4.4.)
Let $\Omega\subset \mathbb{R}^{N}$ be an open bounded domain with Lipschitz boundary. Suppose $A\in N(\Omega)$ satisfies $(P1)$-$(P5)$ and $(P^{1}_{\partial})$.
There exist two constants $\epsilon>0$ and $\delta>0$ such that
$A_{*}^{\frac{N-1}{N}-\frac{\epsilon}{2}}\in\Delta_{2}(\overline{\Omega}_{\delta})$ and $A_{*}^{\frac{N-1}{N}-\epsilon}\in N(\overline{\Omega}_{\delta})$. Then for any $\theta\in N(\partial\Omega)$ satisfies
\begin{equation}\label{2.1}
  \theta(x,t)\leq A_{*}^{\frac{N-1}{N}-\epsilon}(x,t) \text{ for any }t>1\text{ and }x\in\partial\Omega,
\end{equation}
the boundary trace embedding $W^{1,A}(\Omega)\hookrightarrow L^{\theta}(\partial\Omega)$ is compact.
\end{thm}

We claim that Condition (\ref{2.1}) is strong for the compact embedding in Theorem \ref{liu2}. In fact, some more weaker assumptions are needed instead in the theorem. We give a stronger result of the compact trace embedding theorem in Musielak-Sobolev sapces as follows.

\begin{thm}\label{New_Bdtr}
Suppose $\Omega\subset \mathbb{R}^{N}$ is a bounded domain with Lipschitz boundary.  $A\in N(\Omega)$ satisfies $(A_2)$, $(P1)$-$(P5)$ and $(P^{1}_{\partial})$.
Then for any continuous $\theta\in N(\partial\Omega)$ satisfying $\theta(x,t)\in(0,+\infty)$  for any $x\in\partial \Omega$, $t>0$ and $\theta\ll A_{*}^{\frac{N-1}{N}}$,
the boundary trace embedding $W^{1,A}(\Omega)\hookrightarrow L^{\theta}(\partial\Omega)$ is compact.
\end{thm}

Before we give a proof of Theorem \ref{New_Bdtr}, we need a lemma in \cite{1Fan:12} or \cite{Adams:03}.
\begin{lem}\label{cedu}(See Lemma 3.9 in \cite{1Fan:12}.)
Let $A,B\in N(\Omega)$. Suppose that $B:\overline{\Omega}\times [0,+\infty)\to[0,+\infty)$ is continuous, $B\ll A$ and for $x\in\overline{\Omega}$, $t\in(0,+\infty)$, $B(x,t)>0$. If a sequence $\{u_{n}\}$ is bounded in $L^{A}(\Omega)$, and convergent in measure on $\Omega$, then it is convergent in norm in $L^{B}(\Omega)$.
\end{lem}

\noindent\textbf{The proof of Theorem \ref{New_Bdtr}.}
By Propositon \ref{embedd}, we know $W^{1,A}(\Omega)\hookrightarrow W^{1,1}(\Omega)$. And by the boundary compact embedding in Lebesgue-Sobolev space, we have $W^{1,A}(\Omega)\hookrightarrow W^{1,1}(\Omega)\hookrightarrow\hookrightarrow L^{1}(\partial \Omega)$. Then for any bounded sequence in $W^{1,A}(\Omega)$, there exists a subsequence such that it is convergent in $L^{1}(\partial \Omega)$. Then by the Lebesgue theorem, the subsequence is convergent in measure on $\partial \Omega$. At the same time by Theorem \ref{liu1} the subsequence we have found is bounded in $L^{{A_{*}}^{\frac{N-1}{N}}}(\partial\Omega)$. Then by Lemma \ref{cedu}, this subsequence is convergent in $L^{\theta}(\partial \Omega)$.

\begin{rem}
It is clear that $\theta(x,t)\leq A_{*}^{\frac{N-1}{N}-\epsilon}(x,t)$ for any $t>1$ and $x\in\partial\Omega$ implies $\theta\ll A_{*}^{\frac{N-1}{N}}$. Then it is clear that Theorem \ref{New_Bdtr} is stronger than Theorem \ref{liu2}.
\end{rem}

Next we claim that not only variable exponent Sobolev spaces satisfy the conditions in Theorem \ref{New_Bdtr} (see the following \textbf{Example \ref{ex1}}), but also some more complex case also satisfies conditions in Theorem \ref{New_Bdtr} (see the following \textbf{Example \ref{ex2}}). And the current existing theories in mathematical literature can not cover the case in \textbf{Example \ref{ex2}}.

\begin{ex} \label{ex1}
Let $p\in C^{1-0}(\overline\Omega)$ and $1< q\leq p(x)\leq
p_+:=\text{sup}_{x\in{\overline\Omega}}p(x)<n$ ($q\in \mathbb{R}$)
for $x\in\overline\Omega$. Define
$A:\overline\Omega\times\mathbb{R}\rightarrow[0,+\infty)$ by
\begin{equation*}
A(x,t)=|t|^{p(x)}.
\end{equation*}
Then it is readily checked that $A$ satisfies $(A_2)$, $(P1)$, $(P2)$ and $(P3)$. It is easy to see that $p\in C^{1-0}(\overline\Omega)$ means
$A\in C^{1-0}(\overline\Omega)$ and for $s>0$
\begin{equation}\label{ExAStar}
A_*^{-1}(x,s)=\frac{np(x)}{n-p(x)}s^{\frac{n-p(x)}{np(x)}}.
\end{equation}
Then $T(x)=+\infty$, which implies that $(P^1_{\partial})$ is satisfied ($\Rightarrow (P4)$ is satisfied).

In additional for $x\in\Omega$
\begin{equation*}
\nabla_x A(x,t)=|t|^{p(x)}(\ln |t|)\nabla p(x).
\end{equation*}
Since for any $\epsilon>0$, $\frac{\ln t}{t^\epsilon}\rightarrow0$ as $t\rightarrow+\infty$, we conclude that there
exist constants $\delta_1<\frac{1}{n}$, $c_1$ and $t_1$ such that
\begin{equation*}
\bigg|\frac{\partial A(x,t)}{\partial x_j}\bigg|\leq c_1A^{1+\delta_1}(x,t)
\end{equation*}
for all $x\in\Omega$ and $t\geq t_1$. Combining by $A\in\Delta_2(\Omega)$, from Proposition 3.1 in \cite{1Fan:12}
it is easy to see that Condition $(P5)$ is satisfied. By now conditions in Theorem \ref{New_Bdtr} are
verified.

By \eqref{ExAStar}, we can see that
\begin{equation*}
A_*(x,t)=\bigg(\frac{n-p(x)}{np(x)}|t|\bigg)^{\frac{np(x)}{n-p(x)}}.
\end{equation*}
Then $A_*^{\frac{n-1}{n}}\in\Delta_2(\Omega)$ and since $p(x)\leq p_+<n$ we can see
there exists a $\epsilon>0$ such that $A_*^{\frac{n-1}{n}-\epsilon}\in N(\overline{\Omega_{\delta}})$.
Then all conditions in Theorem \ref{New_Bdtr} can be satisfied.
\end{ex}

\begin{ex}\label{ex2}
Let $p\in
C^{1-0}(\overline\Omega)$ satisfy $1< p^-\leq p(x)\leq
p_+:=\text{sup}_{x\in{\overline\Omega}}p(x)<n-1$. Define
$A:\overline\Omega\times\mathbb{R}\rightarrow[0,+\infty)$ by
\begin{equation*}
A(x,t)=|t|^{p(x)}\log(1+|t|),
\end{equation*}

It is obvious that $A$ satisfies $(A_2)$, $(P1)$, $(P2)$ and $(P3)$. Pick
$\epsilon>0$ small enough such that $p^++\epsilon<n$. Then for $t>0$
big enough, $A(x,t)\leq c t^{p^++\epsilon}$, which implies that
$T(x)=+\infty$ for all $x\in\overline\Omega$. Thus
$(P^{1}_{\partial})$ is satisfied ($\Rightarrow (P_4)$ is
satisfied). Since $p\in C^{1-0}(\overline\Omega)$ and $A\in
C^{1-0}(\overline\Omega\times\mathbb{R})$, by Proposition 3.1 in
\cite{1Fan:12}, $A_*\in C^{1-0}(\overline\Omega\times\mathbb{R})$.
Combining $A\in\Delta_2(\Omega)$, it is easy to see that condition
$(P5)$ is satisfied. Then conditions in Theorem \ref{New_Bdtr}
are verified.
\end{ex}

\section{Variational structures in Musielak-Sobolev spaces}\label{Sec3}
Firstly, we give some assumptions of the differential operators in Musielak-Sobolev spaces. 

We say that $A:\overline{\Omega}\times[0,+\infty)\to[0,+\infty)$ satisfies Condition ($\mathscr{A}$), if both $A$ and $\widetilde{A}\in N(\Omega)$ are locally integrable and satisfy $\Delta_{2}(\Omega)$, $A$ satisfies $(A_2)$, (P1)-(P5), $(P^{1}_{\partial})$ in Section \ref{Sec2} and $(A_{\infty})$ below. 
%and there exist $\delta,\epsilon>0$ such that $A_{*}^{\frac{N-1}{N}-\epsilon}\in N(\overline{\Omega_{\delta}})$.
\begin{enumerate}
\item[($A_{\infty}$)] There exists a continous function $A_{\infty}:[0,+\infty)\to[0,+\infty)$, such that $A_{\infty}(\alpha)\to+\infty$ as
$\alpha\to+\infty$ and
\begin{equation*}
A(x,\alpha t)\geq A_{\infty}(\alpha)\alpha A(x,t)
\end{equation*}
for any $x\in\Omega$, $t\in[0,+\infty)$ and $\alpha>0$.
\end{enumerate}

We always assume that Condition ($\mathscr{A}$) holds in this paper. And by Proposition \ref{p2.6}, the Musielak-Sobolev space $W^{1,A}(\Omega)$ is a separable and reflexive Banach space under our assumptions in ($\mathscr{A}$).

We say that a function $a_{1}:\Omega\times \mathbb{R}\to \mathbb{R}$ satisfies ($A1$), if the following conditions ($A_{11})-(A_{14}$) are satisfied:
\begin{enumerate}
\item[($A_{11}$)] $a_{1}(x,|t|):$ $\Omega\times \mathbb{R}\to \mathbb{R}$ is a \textit{Carath\'{e}odory} function;
\item[($A_{12}$)] There exists a positive constant $b_{1}$ such that
$|a_{1}(x,|t|)|t^{2}\leq b_{1}A(x,t)$ for $x\in\Omega$ and $t\in\mathbb{R}$;
\item[($A_{13}$)] There exists a positive constant $b_{2}$ such that
$a_{1}(x,|t|)t\geq b_{2}a(x,t)$ for $x\in\Omega$ and $t\in\mathbb{R}$;
\item[($A_{14}$)] $(a_{1}(x,|\xi|)\xi-a_{1}(x,|\eta|)\eta)(\xi-\eta)>0$ for $x\in\Omega$ and $\xi,\eta \in \mathbb{R}^{N}$.
\end{enumerate}

We say that a function $a_{0}:\Omega\times \mathbb{R}\to \mathbb{R}$ satisfies $(A0)$, if the following conditions $(A_{01})$-$(A_{04})$ are satisfied:
\begin{enumerate}
\item[($A_{01}$)] $a_{0}(x,t): \Omega\times \mathbb{R}\to \mathbb{R}$ is a \textit{Carath\'{e}odory} function;
\item[($A_{02}$)] There exists a positive constant $b_{1}$ such that
 $|a_{0}(x,t)|t\leq b_{1}A(x,t)$ for $x\in\Omega$ and $t\in\mathbb{R}$;
\item[($A_{03}$)] There exists a positive constant $b_{2}$ such that $a_{0}(x,t)\geq b_{2}a(x,t)$ for $x\in\Omega$ and $t\in\mathbb{R}$;
\item[($A_{04}$)] $(a_{0}(x,t_{1})-a_{0}(x,t_{2}))(t_{1}-t_{2})> 0$ for $x\in\Omega$ and $t_{1},t_{2}\in \mathbb{R}$.
\end{enumerate}

Assume $X$ is a real reflexive Banach space. We call the map $T:X\to X^{*}$ is of type $S_{+}$, if for any weakly convergent sequence $u_{n}\rightharpoonup u_{0}$ in $X$ satisfying $\overline{\lim\limits_{n\to\infty}}\langle T(u_{n}),u_{n}-u_{0}\rangle\leq 0$, we have $u_{n}\rightarrow u_{0}$ in $X$.

Define $\mathfrak{A}': W^{1,A}(\Omega)\to (W^{1,A}(\Omega))^{*}$ by
\begin{equation*}
\langle\mathfrak{A}'(u),v\rangle=\int\limits_{\Omega}a_{1}(x,|\nabla u|)\nabla u \nabla v+a_{0}(x,u)v\,\mathrm{d}x.
\end{equation*}
\begin{rem} \label{p1}
By Theorem 2.1 and Theorem 2.2 in \cite{Fan:12}, $\mathfrak{A}'$ is a bounded, continuous, coercive, strictly monotone homeomorphic and $S_{+}$ type operator.
\end{rem}

%\begin{proposition} (See \cite{Fan:12}.)
%Let $A$, $a_{1}$ and $a_{0}$ satisfy ($\mathscr{A}$), $(\overline{A1})$ and $(\overline{A0}) $ respectively. Then the mapping $\mathfrak{A}':W^{1,A}(\Omega)\to (W^{1,A}(\Omega))^{*}$ is bounded, continuous, monotone and coercive. If both $a_{0}$ and $a_{1}$ are strictly monotone in the second variable, that is, $A_{04}$ and $A_{14}$ are replaced by \\
%$(A_{04}^{+})$ $(a_{0}(x,t_{1})-a_{0}(x,t_{2}))(t_{1}-t_{2})> 0$ for $x\in\Omega$ and $t_{1}\neq t_{2}\in \mathbb{R}$;\\
%$(A_{14}^{+})$ $(a_{1}(x,|\xi|)\xi-a_{1}(x,|\eta|)\eta)(\xi-\eta)> 0$ for $x\in\Omega$ and $\xi\neq\eta \in \mathbb{R}^{N}$.\\
%Then $\mathfrak{A}'$ is a strictly monotone homeomorphism and also a $S_{+}$ type operator.
%\end{proposition}

%We say that $a_{0},a_{1}$ satisfy Condition ($\overline{A0}^+$) and ($\overline{A1}^+$) respectively, if the ($\overline{A0}$)-($A_{04}^{+}$) and ($\overline{A1}$)-($A_{14}^{+}$) hold respectively. In this paper, we always assume  ($\mathscr{A}$), ($\overline{A0}^+$) and $(\overline{A1}^+)$ hold.

\section{Existence results of weak solutions}\label{Sec4}
In this section, we consider the existence of solutions to the following Neumann boundary value problem
\begin{equation}\label{1.1}
\begin{cases}
  -\text{div}(a_{1}(x,|\nabla u|)\nabla u)+a_{0}(x,u)=0,&\text{in }\Omega,\\
  a_{1}(x,|\nabla u|)\frac{\partial u}{\partial n}=g(x,u)&\text{on }\partial\Omega,
  \end{cases}
\end{equation}
where $\Omega$ is a bounded smooth domain in $\mathbb{R}^{N}$, $a_{0},a_1:\Omega\times \mathbb{R}\to \mathbb{R}$ are functions satisfying $(A0)$ and $(A1)$, the $A(x,t)$ in $(A0)$ and $(A1)$ satisfies ($\mathscr{A}$), and $\frac{\partial u}{\partial n}$ is the outward unit normal derivative on $\partial \Omega$.

We say that $u$ is a weak solution for (\ref{1.1}) if
\begin{equation*}
\int_{\Omega}a_{1}(x,|\nabla u|)\nabla u\nabla \varphi+a_{0}(x,u)\varphi \,\mathrm{d}x=\int_{\partial\Omega}g(x,u)\varphi \,\mathrm{d}\sigma,\quad\forall \varphi\in W^{1,A}(\Omega).
\end{equation*}

\begin{thm}\label{th1}
If $g(x,u)=g(x)$ and $g\in L^{\widetilde{A_{*}^{\frac{N-1}{N}}}}(\partial \Omega)$, then \eqref{1.1} has a unique weak solution.
\end{thm}

Theorem \ref{th1} follows from Remark \ref{p1} and Proposition \ref{p2.6} immediately.

Next, we assume the following conditions on $g$,
\begin{enumerate}
\item[$(G1)$] $\Psi\in N(\partial \Omega)\cap\Delta_{2}(\partial\Omega)$ with $\Psi\ll A_{*}^{\frac{N-1}{N}}$, $\psi(x,s):=\frac{\partial\Psi(x,s)}{\partial s}$ exists, and there exists a constant $K_{1}>0$ such that
\begin{equation*}
|g(x,t)|\leq K_{1}\psi(x,|t|)+h(x),\quad x\in\partial\Omega,t\in \mathbb{R},
\end{equation*}
where $0\leq h\in L^{\widetilde\Psi}(\partial \Omega)$.
\item[$(G2)$] There exist $\theta>0, C>0$, and $M>0$, such that for $s\geq M$,
\begin{equation*}
0<\theta G(x,s)\leq sg(x,s),\,\forall x\in\partial\Omega,
\end{equation*}
where $\theta>C_{a}:=\max\{\overline{\lim\limits_{s\to\infty}}\frac{
a_{1}(x,|s|)s^{2}}{A_{1}(x,s)},\overline{\lim\limits_{s\to\infty}}\frac{
a_{0}(x,s)s}{A_{0}(x,s)}\}$.
\end{enumerate}

\begin{rem}
In fact, by assumptions $(A_{02})$, $(A_{03})$, $(A_{12})$ and $(A_{13})$, we have $C_{a}\in[1,\frac{b_{1}}{b_{2}}]$. The Condition $(G2)$ is the \textit{generalized Ambrosetti-Rabinowitz condition} for the nonlinear elliptic operator.
\end{rem}

\begin{lem}\label{l2.1}
Let $(G1)$ hold and
set $\mathfrak{G}(u):=\int\limits_{\partial\Omega}G(x,u)\,\mathrm{d}\sigma$. Then $\mathfrak{G}$ is sequentially weakly continuous in $W^{1,A}(\Omega)$ and $\mathfrak{G}':W^{1,A}(\Omega)\to (W^{1,A}(\Omega))^{*}$ is a completely continuous operator.
\end{lem}

\begin{proof}
Suppose $u_{n}\rightharpoonup u$ weakly in $W^{1,A}(\Omega)$, to prove $\mathfrak{G}$ is sequentially weakly continuous, it is sufficient to prove that
\begin{equation*}
\lim\limits_{n\to\infty}\int\limits_{\partial \Omega}G(x,u_{n})-G(x,u)\,\mathrm{d}\sigma=0.
\end{equation*}
In fact,
\begin{align*}
\bigg|\int\limits_{\partial \Omega}G(x,u_{n})-G(x,u)\,\mathrm{d}\sigma\bigg|&\leq\int\limits_{\partial \Omega}|G(x,u_{n})-G(x,u)|\,\mathrm{d}\sigma\\
&\leq K_{1}\int\limits_{\partial \Omega}|\Psi(x,|u_{n}|)-\Psi(x,|u|)|\,\mathrm{d}\sigma\\
&\quad\quad+2\|u_{n}-u\|_{L^{\Psi}(\partial\Omega)}
\|h\|_{L^{\widetilde\Psi}(\partial\Omega)}.
\end{align*}
From assumption $(G1)$, the compact embedding $W^{1,A}(\Omega)\hookrightarrow L^{\Psi}(\partial \Omega)$ and Proposition 2.3 in \cite{Liu:15}, we know that the right hand side of the above inequality tends to $0$.

Since $\langle\mathfrak{G}'(u),v\rangle=\int\limits_{\partial \Omega}g(x,u)v\,\mathrm{d}\sigma$, by the compact embedding $W^{1,A}(\Omega)\hookrightarrow L^{\Psi}(\partial \Omega)$, $\mathfrak{G}'$ is a completely continuous operator.
\end{proof}

Under assumption $(A_{\infty})$, the following assumptions on $A$ and $\Psi$ will be used.
\begin{enumerate}
\item[($\Psi_{\infty}$)] $\Psi(x,\alpha t)\leq\Psi_{\infty}(\alpha)\alpha \Psi(x,t)$, in which $\Psi_{\infty}: \mathbb{R}^{+}\to\mathbb{R}^{+}$ is a nondecreasing function
and satisfies
\begin{equation*}
\Psi_{\infty}(\alpha)\to\infty,\text{ as }\alpha\to \infty
\end{equation*}
and 
\begin{equation}\label{Psi-infty}
\lim\limits_{t\to\infty}\frac{\Psi_{\infty}(Ct)}{A_{\infty}(t)}=0,\,\forall\, C>0;
\end{equation}
\item[($\Psi_{0}$)] $\Psi(x,\alpha t)\leq\Psi_{0}(\alpha)\alpha \Psi(x,t)$, in which $\Psi_{0}: \mathbb{R}^{+}\to\mathbb{R}^{+}$ is a nondecreasing function
and satisfies
\begin{equation*}
\frac{A_{\infty}(\alpha)}{\Psi_{0}(C\alpha)}\to \infty,\text{ as }\alpha\to0,\,\forall C>0;
\end{equation*}
\item[($A_{0}$)] There exists a function $A_{0}:(0,+\infty)\to (0,+\infty)$ such that $A(x,\alpha t)\leq A_{0}(\alpha)\alpha A(x,t)$;
\item[($G3$)] $g(x,t)=-g(x,-t)$, for $x\in\partial\Omega, t\in \mathbb{R}$;
\item[($G4$)] There is a function $G_{0}:(0,+\infty)\to(0,+\infty)$ satisfying
\begin{equation*}
G(x,\alpha t)\geq G_{0}(\alpha)\alpha G(x,t)\text{ and }
\frac{G_{0}(\alpha)}{A_{0}(\alpha)}\to \infty,\text{ as }\alpha\to0
\end{equation*}
for $x\in\partial\Omega, t\in \mathbb{R}, \alpha>0$.
\end{enumerate}

\begin{thm}\label{th2}
If $g(x,u)$ satisfies Condition $(G1)$ and $(\Psi_{\infty})$, then (\ref{1.1}) has a weak solution.
\end{thm}

Define the functional $J(u)$ corresponding (\ref{1.1}) by
\begin{equation*}
J(u)=\int\limits_{\Omega}A_{1}(x,|\nabla u|)+A_{0}(x,u)\,\mathrm{d}x
-\int\limits_{\partial\Omega}G(x,u)\,\mathrm{d}\sigma,
\end{equation*}
in which $A_{1}(x,t)=\int_{0}^{t}a_{1}(x,s)s\,\mathrm{d}s$, $A_{0}(x,t)=\int_{0}^{t}a_{0}(x,s)\,\mathrm{d}s$, $G(x,t)=\int_{0}^{t}g(x,s)\,\mathrm{d}s$. Under the assumptions in Theorem \ref{th2}, $J(u)$ is well defined in $W^{1,A}(\Omega)$ and $J\in C^{1}$.
\begin{rem}
Under assumption $(G1)$, by Lemma \ref{l2.1} and Remark \ref{p1}, we know that $J'$ is of type $S_{+}$.
\end{rem}

\noindent\textbf{The proof of Theorem \ref{th2}.}
By $(A_{03})$, $(A_{13})$, $(G1)$ and $(\Psi_{\infty})$, we obtain
\begin{equation*}
\begin{aligned}
J(u)&=\int\limits_{\Omega}A_{1}(x,|\nabla u|)+A_{0}(x,u)\,\mathrm{d}x
-\int\limits_{\partial\Omega}G(x,u)\,\mathrm{d}\sigma\\
&\geq b_{2}\int\limits_{\Omega}A(x,|\nabla u|)+A(x,| u|)\,\mathrm{d}x-\int\limits_{\partial\Omega}\Psi(x,u)+|u(x)|h(x)\,\mathrm{d}\sigma\\
&\geq b_{2}A_{\infty}(\|u\|_{\rho})\|u\|_{\rho}\int\limits_{\Omega}A(x,\frac{|\nabla u|}{\|u\|_{\rho}})+A(x,\frac{|u|}{\|u\|_{\rho}})\,\mathrm{d}x\\
&\quad\quad-K_{1}\Psi_{\infty}(\|u\|_{\Psi})\|u\|_{\Psi}
\int\limits_{\partial\Omega}\Psi(x,\frac{|u|}{\|u\|_{\Psi}})\,\mathrm{d}\sigma-\|u\|_{\Psi}\|h\|_{\Psi}\\
&\geq b_{2}A_{\infty}(\|u\|_{\rho})\|u\|_{\rho}
-C_{1}\Psi_{\infty}(c\|u\|_{\rho})\|u\|_{\rho}-C\|u\|_{\rho}\|h\|_{\Psi}\\
&=(b_{2}A_{\infty}(\|u\|_{\rho})-C_{1}\Psi_{\infty}(c\|u\|_{\rho})
-C\|h\|_{\Psi})\|u\|_{\rho}\geq \frac{b_{2}}{2}A_{\infty}(\|u\|_{\rho})\|u\|_{\rho}\to\infty, 
\end{aligned}
\end{equation*}
as $\|u\|_{\rho}\to\infty$. The last inequality sign follows from \eqref{Psi-infty}.
By Lemma \ref{l2.1} and Remark \ref{p1}, we know that $J$ is weakly lower semicontinuous in $W^{1,A}(\Omega)$. Then $J$ admits a minimum in $W^{1,A}(\Omega)$ which is a weak solution for (\ref{1.1}). The proof is completed.

\begin{lem}\label{ps}
Under the assumptions $(G1)$ and $(G2)$, $J$ satisfies the $P.S$. condition.
\end{lem}

\begin{proof}
Let $\{u_{n}\}_{n\in\mathbb{N}}$ be a $P.S$. sequence of $J$, i. e., $J(u_{n})\to C$, $J'(u_{n})\to 0$. Fixing $\beta\in (\frac{1}{\theta},\frac{1}{C_{a}+\varepsilon})$ with $\varepsilon>0$ small enough, by the definition of $C_{a}$, there exists $C_{\varepsilon}>0$, such that
\begin{align*}
C+1+\|u_{n}\|\geq& J(u_{n})-\beta\langle J'(u_{n}),u_{n}\rangle \\
 =&\int\limits_{\Omega}A_{1}(x,|\nabla u_{n}|)+A_{0}(x,|u_{n}|)\,\mathrm{d}x-
 \beta\int\limits_{\Omega}a_{1}(x,|\nabla u_{n}|)|\nabla u_{n}|^{2}\\&\quad\quad+a_{0}(x,|u_{n}|)u_{n}\,\mathrm{d}x+\int\limits_{\partial\Omega}\beta g(x,u_{n})u_{n}-G(x,u_{n})\,\mathrm{d}\sigma\\
 \geq& (1-\beta (C_{a}+\varepsilon))\int\limits_{\Omega}A_{1}(x,|\nabla u_{n}|)+A_{0}(x,|u_{n}|)\,\mathrm{d}x+\\&\quad\quad
\int\limits_{\partial\Omega}\beta g(x,u_{n})u_{n}-G(x,u_{n})\,\mathrm{d}\sigma-C_{\varepsilon}\\
 \geq& (1-\beta (C_{a}+\varepsilon))b_{2}A_{\infty}(\|u_{n}\|)\|u_{n}\|-C_{\varepsilon}.
\end{align*}
From the above inequality, it is easy to see that $\{u_{n}\}$ is bounded in $W^{1,A}(\Omega)$. Then there exists a subsequence of  $\{u_{n}\}$ (still denoted by $\{u_{n}\}$) such that $u_{n}\rightharpoonup u_{0}$ in $W^{1,A}(\Omega)$. So we conclude that $\langle J'(u_{n})-J'(u_{0}),u_{n}-u_{0}\rangle\to 0$. Since $J'$ is a $S_{+}$ type operator, we  can conclude that $u_{n}\to u_{0}$.
\end{proof}

\begin{thm}\label{th3}
Under the assumptions ($G1$), ($G2$), ($\Psi_{0}$) and $h\equiv0$ in ($G1$), (\ref{1.1}) has a nontrivial solution.
\end{thm}
\begin{proof}
It is sufficient to verify that $J$ satisfies conditions in Mountain Pass Lemma (see \cite{Willem}). 

By ($G1$) and ($\Psi_{0}$), we obtian
\begin{align*}
J(u)\geq& b_{2}A_{\infty}(\|u\|)\|u\|-K_{1}\int\limits_{\partial\Omega}
\Psi(x,\frac{|u|}{\|u\|_{\Psi}}\|u\|_{\Psi})\,\mathrm{d}\sigma\\
\geq& b_{2}A_{\infty}(\|u\|)\|u\|-K_{1}\Psi_{0}( \|u\|_{\Psi})\|u\|_{\Psi}\\
\geq&\big[ b_{2}A_{\infty}(\|u\|)-C_{1}\Psi_{0}( C\|u\|) \big]\|u\|.
\end{align*}
By the assumption $(\Psi_{0})$, there exists a small $\gamma>0$ such that $b_{2}A_{\infty}(\|u\|)-C_{1}\Psi_{0}( C\|u\|) >0$ for $0<\|u\|\leq \gamma$. Hence, for $\|u\|=\gamma$, there exists a $\delta>0$ such that $J(u)\geq \delta>0$.

Next we will show that there exists a $w\in W^{1,A}(\Omega)$ with $\|w\|>\rho$ such that $J(w)<0$. In fact, for any $w\in W^{1,A}(\Omega)\cap W^{1,C_{a}+\varepsilon}(\Omega)$ with $\varepsilon$ small enough such that $C_{a}+\varepsilon<\theta$, we can get
\begin{align*}
J(tw)=&\int\limits_{\Omega}A_{1}(x,|t\nabla w|)+A_{0}(x,tw)\,\mathrm{d}x
-\int\limits_{\partial\Omega}G(x,tw)\,\mathrm{d}\sigma\\
\leq& C_{1}t^{C_{a}+\varepsilon}\int\limits_{\Omega}|\nabla w|^{C_{a}+\varepsilon}+|w|^{C_{a}+\varepsilon}\,\mathrm{d}x
-C_{2}t^{\theta}\int\limits_{\partial\Omega}w^{\theta}\,\mathrm{d}\sigma+C_{\varepsilon}.
\end{align*}
Let $t\to \infty$, we have $J(tw)\to-\infty$, which completes the proof.
\end{proof}

Denote $X:=W^{1,A}(\Omega)$. Then under our assumptions in ($\mathscr{A}$), there exists $\{e_{j}\}_{j=1}^{\infty}\subset X$ and $\{e_{j}^{*}\}_{j=1}^{\infty}\subset X^{*}$ such that
\begin{align*}
 \langle e_{i}^{*},e_{j}\rangle=\begin{cases}&1,\quad i=j,\\
 &0,\quad i\neq j;
 \end{cases}
\end{align*}
and
\begin{center}
$X=\overline{\text{span}\{e_{j},j=1,2,\cdots\}}$, $X^{*}=\overline{\text{span}\{e^{*}_{j},j=1,2,\cdots\}}$.
\end{center}
For $k=1,2,\cdots$, denote
\begin{equation*}
X_{j}=\text{span}\{e_{j}\}, Y_{k}=\bigoplus\limits_{j=1}^{k}X_{j},
Z_{k}=\overline{\bigoplus\limits_{j=k}^{\infty}X_{j}}.
\end{equation*}

\begin{lem}\label{l4.4}
Let $\Psi \ll A_{*}^{\frac{N-1}{N}}$, and denote
\begin{equation*}
\beta_{k}=\sup\bigg\{\int\limits_{\partial\Omega}\Psi(x,|u|)+h(x)|u|\,\mathrm{d}\sigma:\|u\|\leq \vartheta,u\in Z_{k}\bigg\},
\end{equation*}
in which $\vartheta$ is any given positive number. Then $\lim\limits_{k\to\infty}\beta_{k}=0$.
\end{lem}

\begin{proof}
It is obvious that $0<\beta_{k+1}\leq\beta_{k}$. Then there exists a $\beta\geq 0$ such that $\beta_{k}\to \beta$ as $k\to\infty$. By the definition of $\beta_{k}$, there exist $u_{k}\in Z_{k}$ and $\|u_{k}\|\leq \vartheta$ such that $0\leq\int\limits_{\partial\Omega}\Psi(x,|u_{k}|)+h(x)|u_{k}|\,\mathrm{d}\sigma-\beta<\frac{1}{k}$. Then there exists a subsequence of $\{u_{k}\}$ (still denoted by $\{u_{k}\}$) satisfying $u_{k}\rightharpoonup u$ in $W^{1,A}(\Omega)$ and $u_{k}\to u$ in $L^{\Psi}(\partial\Omega)$. By the definition of $Z_{k}$ we obtain
\begin{equation*}
\langle e_{j}^{*},u\rangle=\lim\limits_{k\to\infty}\langle e_{j}^{*},u_{k}\rangle=0,\quad\forall j=1,2,\cdots.
\end{equation*}
Then $u=0$, which means $u_{k}\to0$ in $L^{\Psi}(\partial\Omega)$. $\beta_{k}\to 0$ follows by Proposition 2.3 in \cite{Liu:15} and the assumption $h\in L^{\widetilde\Psi}(\partial\Omega)$.
\end{proof}

\begin{thm}
Let $(G1)$, $(G3)$, $(G4)$ and $(\Psi_{\infty})$ hold. Then \eqref{1.1} has a sequence of solutions $\{u_{n}\}_{n\in \mathbb{N}}$ with negative energy and  $I(u_{n})\to 0$.
\end{thm}

\begin{proof}
By Condition ($G1$) and ($\Psi_{\infty}$), the functional $J(u)$ is coercive and satisfies the P.S. condition.\\
By ($G3$), $J(u)$ is an even functional.
 Denote by $\gamma(U)$ the genus of $U\in\Sigma:=\{U\subset W^{1,A}(\Omega)\backslash{\{0\}}: U \text{ is compact and }U=-U\}$ (see \cite{Struwe, Willem}). Set
\begin{equation*}
\begin{aligned}
\Sigma_{k}&:=\{U\in \Sigma: \gamma(U) \geq k\},\, k = 1, 2,\cdots,\\
c_{k}&:=\inf\limits_{U\in \Sigma_{k}}\sup\limits_{u\in U}I(u),\,k = 1, 2,\cdots.
\end{aligned}
\end{equation*}
Then it is clear that
\begin{equation*}
-\infty<c_{1}\leq c_{2}\leq\cdots\leq c_{k}\leq c_{k+1}\leq\cdots.
\end{equation*}

Next we will show that $c_{k}<0$ for every $k\in\mathbb{N}$. In fact, selecting $E_{k}$ be the $k$-dimensional subspaces of $W^{1,A}(\Omega)$ with $E_{k}\subset C^{\infty}(\overline{\Omega})$ such that $u|_{\partial \Omega}\not\equiv 0$
for all $u\in E_{k}\backslash\{0\}$, by $(G4)$ and $(A_{0})$, we have for $u\in E_{k}, \|u\|=1$,
\begin{align*}
J(tu)=&\int\limits_{\Omega}A_{1}(x,|t\nabla u|)+A_{0}(x,tu)\,\mathrm{d}x
-\int\limits_{\partial\Omega}G(x,tu)\,\mathrm{d}\sigma\\
\leq& A_{0}(t)t\int\limits_{\Omega}A_{1}(x,|\nabla u|)+A_{0}(x,u)\,\mathrm{d}x
-G_{0}(t)t\int\limits_{\partial\Omega}G(x,u)\,\mathrm{d}\sigma,\\
:=& A_{0}(t)t a_{k}-G_{0}(t)t g_{k},
\end{align*}
where
\begin{equation*}
\begin{aligned}
a_{k}&=\sup\bigg\{\int\limits_{\Omega}A_{1}(x,|\nabla u|)+A_{0}(x,u)\,\mathrm{d}x:\,u\in E_{k},\|u\|=1\bigg\},\\
g_{k}&=\inf\bigg\{\int\limits_{\partial\Omega}G(x,u)\,\mathrm{d}\sigma:\,u\in E_{k},\|u\|=1\bigg\}.
\end{aligned}
\end{equation*}
Observing that $0<a_{k}<+\infty$ and $g_{k}>0$ since $E_{k}$ is of finite dimensional, by the assumption $(G4)$, there exists positive constants $\rho$ and $\varepsilon$ such that
\begin{equation*}
J(\rho u)<-\varepsilon\text{ for }u\in E_{k},\|u\|=1.
\end{equation*}
Therefore, if we set $S_{\rho,k}=\{u\in E_{k}:\|u\|=\rho\}$, then $S_{\rho,k}\subset J^{-\varepsilon}$. Hence, by the monotonicity of the genus
\begin{equation*}
\gamma(J^{-\varepsilon}) \geq \gamma(S_{\rho,k})=k,
\end{equation*}
By the minimax argument, each $c_{k}$ is a critical value of $J$ (see \cite{Struwe}),  hence there is a sequence of weak solutions $\{\pm u_{k},k=1,2,\cdots\}$ of the problem (\ref{1.1}) such that $J(u_{k})=c_{k}<0$.

Next we will show that $c_{k}\to 0$.  In fact, by $(G1)$ and Lemma \ref{l4.4}, for $u\in Z_{k},\|u\|\leq \vartheta$, we obtain
\begin{align*}
J(u)&\geq \int\limits_{\Omega}A_{1}(x,|\nabla u|)+A_{0}(x,u)\,\mathrm{d}x-\int\limits_{\partial\Omega}\Psi(x,|u|)+h(x)|u|\,\mathrm{d}\sigma\\
&\geq -\int\limits_{\partial\Omega}\Psi(x,|u|)+h(x)|u|\,\mathrm{d}\sigma\geq -\beta_{k}.
\end{align*}
It is clear that for $U\in \Sigma_{k}$, we have $\gamma(U)\geq k$. Then according to the properties of genus, we conclude $U\cap Z_{k}\neq\emptyset$ for $U\in \Sigma_{k}$. Then
\begin{equation*}
\sup\limits_{u\in U\in\Sigma_{k}}J(u)\geq -\beta_{k}.
\end{equation*}
By the definition of $c_k$, it is easy to see that $0\geq c_{k}\geq-\beta_{k}$, which implies $c_{k}\to 0$ by $\beta_k\rightarrow0$ from Lemma \ref{l4.4}.
\end{proof}

\section{Acknowledgments}

The research is partly supported by the National Natural Science
Foundation of China (NSFC 11501268 and 11471147) and the Fundamental Research Funds for the Central Universities (lzujbky-2016-101). The authors express thanks to professor Peihao Zhao for his valuable suggestions on the paper.

\renewcommand{\baselinestretch}{0.1}
\bibliographystyle{plain}
%\bibliography{Ref}

\end{document}